\documentclass[12pt,twoside]{amsart} 
\usepackage{amssymb,color,tikz}
\usetikzlibrary{matrix} \nonstopmode \textwidth=16.00cm

\usepackage[utf8]{inputenc}
\usepackage[T1]{fontenc}
\usepackage{tikz-cd}
\usepackage{lipsum}
\usepackage{adjustbox}
\usepackage{graphicx,color}
\usepackage{comment}
\usepackage{xcolor}
\usepackage{centernot}
\usepackage{mathtools}
\usepackage{stmaryrd}
\usepackage{blkarray}
\usepackage{amsmath}
\usepackage{nicematrix}
\usepackage{enumitem}
\usepackage{yhmath}
\usepackage[backend=biber,style=alphabetic,maxnames=10]{biblatex} 
\usepackage[all]{xy}

\numberwithin{equation}{section} 

\newcommand {\PP}{\mathbb{P}}

\newcommand{\D}{\mathcal{D}}
\DeclareMathOperator{\Soc}{Soc}

\def\cocoa{{\hbox{\rm C\kern-.13em
      o\kern-.07em C\kern-.13em o\kern-.15em A}}}
\newtheorem{theorem}{Theorem}[section]

\newtheorem{proposition}[theorem]{Proposition}
\newtheorem{corollary}[theorem]{Corollary}

 \theoremstyle{definition}
\newtheorem{definition}[theorem]{Definition} \theoremstyle{remark}
\newtheorem{remark}[theorem]{Remark}
\newtheorem{example}[theorem]{Example}
\newtheorem{notation}[theorem]{Notation}

\DeclareMathOperator{\Ann}{Ann}
\DeclareMathOperator{\Hom}{Hom}

\DeclareMathOperator{\Hess}{Hess}
\DeclareMathOperator{\Ext}{Ext}

\usepackage{hyperref}
\definecolor{MyDarkGreen}{cmyk}{0.7,0,1,0}

\newcommand{\h}[1]{\-\mbox{-#1}}

\begin{document}

\title{Jordan type of Full Perazzo algebras}

\author[P.\ Macias Marques]{Pedro Macias Marques}
\address{Universidade de \'{E}vora, Centro de Investiga\c{c}\~{a}o em Matem\'{a}tica e Aplica\c{c}\~{o}es (CIMA), Escola de Ci\^{e}ncias e Tecnologia, Rua Rom\~{a}o Ramalho, 59, P--7000--671 \'{E}vora, Portugal.} \email{pmm@uevora.pt}

 \author[R.\ M.\ Miró\h{Roig}]{Rosa M.\ Miró\h{Roig}}
  \address{Facultat de
  Matem\`atiques i Inform\`atica, Universitat de Barcelona, Gran Via de les
  Corts Catalanes 585, 08007 Barcelona, Spain} \email{miro@ub.edu,  ORCID 0000-0003-1375-6547}

  \author[J.\ Pérez]{Josep Pérez}
 \address{Facultat de
  Matem\`atiques i Inform\`atica, Universitat de Barcelona, Gran Via de les
  Corts Catalanes 585, 08007 Barcelona, Spain} \email{jperezdiez@ub.edu}

\thanks{The second author  was partially supported
by the grant PID2020-113674GB-I00}

\begin{abstract}
In this paper, we compute all possible Jordan types of linear forms in any full Perazzo algebra. In some cases we are also able to compute the corresponding Jordan degree\h{type}, which is a finer invariant.
\end{abstract}

\maketitle

\section{introduction}

For any graded Artinian algebra $A$, if we consider an element $\ell$ in the maximal ideal $\mathfrak{m}=\oplus_{i\ge 1}A_i$, the multiplication map by $\ell$ in a finite\h{length} module $M$ over $A$ is a nilpotent endomorphism. 
The Jordan type of $\ell$ in $M$ gives the sizes of the Jordan blocks of the multiplication by $\ell$, seen as a linear map. 
In the most common setup one takes $M=A$. 
Since \(A\) is Artinian, the Hilbert function of $A$ is a finite sequence of positive integers $h=(h_0,\dots,h_d)$, where $h_i=\dim_K A_i$, called the \mbox{$h$\h{vector}} of \(A\).
It turns out that, for a given \mbox{$h$\h{vector}} there are many configurations of the possible Jordan types that can occur. 
In \cite{IMMM} 
%A.~Iarrobino, P.~Macias Marques, and C.~McDaniel
A.~Iarrobino, C.~McDaniel, and the first author
give a survey of the subject. They also define another invariant, called \textit{Jordan degree\h{type}} \cite[Definition 2.28]{IMMM}, which is finer than the Jordan type. 
It is important to recall that finding all possible Jordan types, that is, finding all possible canonical Jordan matrices associated to multiplication maps is not always an easy task. 
In some cases, though, one can obtain maximum and minimum possible Jordan types with respect to the dominance partial order (Definition \ref{dominanceOrder}). 
In this paper we study the Jordan type associated to the multiplication map by a linear form $\ell\in A_1$ in a standard graded Artinian Gorenstein algebra $A$. For full Perazzo algebras (cf.\ A.~Cerminara et al.\ 
\cite{CGIZ}, and 
%R.M.~Mir\'{o}-Roig and J.~P\'{e}rez 
the second and third authors'
\cite{MP2}) only one case is known (see N.~Abdallah et al.\  \cite{A}). 
In this work we first compute, as a toy example, the Jordan type of any linear form in the full Perazzo algebra whose Macaulay dual generator is the classic Perazzo form $F=X_0Y_1^2+X_1Y_1Y_2+X_2Y_2^2\in K[X_0,X_1,X_2,Y_1,Y_2]$ (Example \ref{toyExample}). Then, using the toy example as a guide, we compute all possible linear Jordan types in any full Perazzo algebra, except for one case \ref{jordanTypeMin}, for which we give the whole range of possible linear Jordan types, with respect to the dominance partial order (Corollary \ref{corollaryRangeJordanType}). 
It is worth mentioning that the Jordan type of a linear form $\ell$ in an Artinian algebra $A$ determines whether or not $\ell$ is a Lefschetz element for $A$ (see Remark~3 and Proposition~14 in T.~Harima and J.~Watanabe's paper \cite{HW}). This paper aims to contribute to the study of Perazzo algebras  inasmuch as it can help computing all possible linear Jordan types of different families of Perazzo algebras arising from full Perazzo forms.
\section{Background material}\label{prelim}

Throughout this paper $K$ will be an infinite field. 
In this section we fix the notation and recall basic facts on Hilbert functions, Artinian Gorenstein algebras, Lefschetz properties, as well as on Jordan type and full Perazzo algebras needed in the paper.

\subsection{Hilbert functions of Artinian Gorenstein algebras} 
 Given any standard graded $K$\h{algebra} $A=R/I$ where $R=K[x_0,\dots,x_n]$ and $I$ is a homogeneous ideal of $R$,
we denote by $HF_A:\mathbb{Z} \longrightarrow \mathbb{Z}$ with $HF_A(j)=\dim_KA_j$
its \emph{Hilbert function}. 
If $A$ is Artinian, then $(x_0,\dots,x_n)^d\subseteq I$ for some $d\ge 0$ and  the Hilbert function of $A$ is captured in a finite sequence of positive integers $h=(h_0,h_1,\dots ,h_d)$, called \emph{$h$\h{vector}} of \(A\), where $h_i:=HF_A(i)>0$ and $d$ is the last index with this property.
The \emph{Sperner} number $S_A$ of $A$ is the maximum value of the $h$\h{vector}, i.e.\ $S_A=\max_{i\in\{0,\dots,d\}}\{h_i\}$. 
The \textit{socle} of $A$ is defined as $\Soc(A):=(0:\mathfrak{m})_A$, where $\mathfrak{m}$ is the maximal ideal $(x_0,\dots,x_n)$.

Recall that the depth of a ring is bounded above by its Krull dimension. 
Then, the depth of a standard graded Artinian algebra $A = R/I$ is zero, for its Krull dimension is zero, and hence $A$ is a Cohen Macaulay ring with type
$$
r(A)=\dim_K\Ext_R^0(K,A)=\dim_K\Hom_R(K,A)=\dim_K \Soc(A).
$$ 
Thus, a standard graded  Artinian  algebra $A$ is Gorenstein if and only if its socle is one\h{dimensional} \cite[Theorem 3.2.10]{BH}. 
In this situation, since $A_d$ is always contained in $\Soc(A)$, $A_d=\Soc(A)$ and the integer $d$ is called the \emph{socle degree of} $A$. 
Moreover, the natural pairing given by multiplication $A_i\times A_{d-i} \rightarrow A_d \cong K$ is perfect and hence it yields an isomorphism $A_i^* \cong A_{d-i}$. So the $h$\h{vector} for such algebras is symmetric. 

\vskip 2mm
Hereafter, we will assume the characteristic of $K$ to be either zero or greater than $d$.
\begin{definition}\label{defUnimodal} 
The $h$\h{vector} $h=(h_0,h_1,\dots ,h_d)$ of a graded Artinian $K$\h{algebra} is said to be \emph{unimodal} if $h_0\le h_1\le \cdots \le h_j\ge h_{j+1}\ge \cdots \ge h_d$ for some $j$.
\end{definition}

\begin{definition}\label{compressed}
A \textit{compressed} Artinian Gorenstein algebra of codimension $c$ and socle degree \(d\) is one for which the Hilbert function is as big as possible. Then, its Hilbert function is
    $$
    \Bigl( 1,\tbinom{c}{c-1},\dots,\tbinom{\frac{d}{2}+c-1}{c-1},\dots,\tbinom{c}{c-1},1 \Bigr)
    $$
    in case $d$ is even, and of the form
    $$
    \Bigl( 1,\tbinom{c}{c-1},\dots,\tbinom{\frac{d-1}{2}+c-1}{c-1},\tbinom{\frac{d-1}{2}+c-1}{c-1},\dots,\tbinom{c}{c-1},1 \Bigr)
    $$
    if $d$ is odd, thanks to the symmetry of the $h$\h{vector} of any Artinian Gorenstein algebra.    
\end{definition}

\subsection{Artinian Gorenstein algebras and Macaulay\h{Matlis} duality}

In this subsection we quickly recall the construction of the standard graded Artinian Gorenstein algebra associated to a given form. 

Let $R=K[x_0,\ldots,x_N]$ be a polynomial ring and let $\D=K[X_0,\ldots,X_N]$ be a divided power algebra, regarded as an $R$\h{module} with the contraction action \cite[Appendix A]{IK}
\[
x_i\circ X_0^{j_0} \cdots X_i^{j_i} \cdots X_N^{j_N}=\begin{cases}X_0^{j_0} \cdots X_i^{j_i-1} \cdots X_N^{j_N} & \text{if} \ j_i>0\\ 0 & \text{otherwise}.\\ \end{cases}
\]

If $I\subset R$ is a homogeneous
ideal, the \emph{Macaulay's inverse system} $I^{-1}$ for $I$ is the following $R$\h{submodule} of $\D$
\[
I^{-1}:=\{F\in \D, f\circ F=0 \text{ for all } f\in I \}.
\]

Conversely, if $M\subset \D$ is a graded $R$\h{submodule},
then $$\Ann(M):=\{f\in R : f\circ F = 0 \text{ for all } F \in M\}$$
is a homogeneous ideal in $R$. It is well known that there is a bijective correspondence
\[
\begin{array}{ccc}  \{ \text{Homogeneous ideals } I\subset R \} &
\rightleftharpoons & \{ \text{Graded } R\text{\h{submodules} of }\D \} \\
I & \rightarrow & I^{-1} \\ \Ann(M) & \leftarrow & M \end{array} .
\]

Moreover, $I^{-1}$ is a finitely generated $R$\h{module} if and only if $R/I$ is an Artinian ring. 
For any integer $t\ge0$ we have $HF_{R/I}(t) =\dim_K(R/I)_t=\dim_K(I^{-1})_t$. 
As a particular case, an Artinian $K$\h{algebra} $A=R/I$ is Gorenstein with socle degree $d$ if and only if $I=\Ann_R(F)=\{f\in R\mid f\circ F=0\}$ for some homogeneous polynomial $F\in \D_d$. 
This polynomial, called the \emph{Macaulay dual generator} for $A$, is unique up to a scalar multiple. In this context, since $A$ is Gorenstein, we have
\begin{align*}
HF_A(t) &=\dim_K(A_t)=\dim_K(A_{d-t})=\dim_K(I^{-1})_{d-t} \\
        &=\dim_K \left\langle x_0^{i_0}\cdots x_N^{i_N} \circ F  : i_0+\cdots+i_N=t\right\rangle.
\end{align*}

\begin{remark}
When \(K\) is a field of characteristic zero, the action just defined can be thought of as differentiation \cite[Example A.5, Appendix A]{IK}. In this context, and as a direct application of Macaulay\h{Matlis} duality, for any integer $t\ge 0$ the Artinian Gorenstein algebra associated to the \textit{homogeneous symmetric} polynomial 
$$
h_t(X_0,\dots,X_n):=\sum_{i_0+\cdots+i_n=t}X_0^{i_0}\cdots X_n^{i_n}
$$
in the usual polynomial ring \(K[X_0,\ldots,X_n]\) (see \cite[Section 2]{BMMRN}) is a compressed algebra \cite[Theorem 2.11]{BMMRN}, \cite[Theorem 5.6]{GL}. 
Note that in the divided\h{power} ring \({\D=K[X_0,\ldots,X_n]}\) this polynomial takes the form 
$$
h_t(X_0,\dots,X_n)=\sum_{i_0+\cdots+i_n=t}i_0!\cdots i_n!X_0^{i_0}\cdots X_n^{i_n}.
$$
\end{remark}

%%%%%%%%%%%%%%%%%%%%%%%%%%%%%%%%%%%%%%

\subsection{Jordan type of a graded Artinian algebra}
\begin{definition} \label{jordanBasisDefinition}
Let $A$ be a graded Artinian algebra. The \textit{Jordan type} of a linear form $\ell\in A_1$ in $A$ is the partition of $\dim_K A$, denoted $P_{\ell,A}=(p_1,\dots,p_s)$ with $p_1 \ge \cdots \ge p_s$, where $p_1, \dots,p_s$ represent the block sizes in the Jordan canonical form of the multiplication map $\times \ell: A \longrightarrow A$. A basis $\mathcal{B}=\{\ell^iz_k:1\le k\le s, \ 0\le i \le p_k-1\}$ of $A$, as a vector space over $K$, is called a \textit{pre\h{Jordan}} basis. A \textit{string} is a sequence
$$
S_k:=(z_k,\ell z_k,\dots, \ell^{p_k-1}z_k)
$$
of $\mathcal{B}$. A \textit{bead} is an element $\ell^iz_k$ in a string $S_k$. A \textit{Jordan basis} for $\ell$ is a pre\h{Jordan} basis $\mathcal{B}$ satisfying $\ell^{p_k}z_k=0$ for each $k\in\{1,\dots,s\}$. It is well known that one can construct a Jordan basis from any pre\h{Jordan} basis \cite[Lemma 2.2 (iii)]{IMMM}. 
If $\nu_k$ is the degree of $z_k$, then the sequence
\[
\mathcal{S}_{\ell,A}=((p_1,\nu_1),\dots,(p_s,\nu_s))
\]
is an invariant of $(A,\ell)$ \cite[Lemma 2.2(iv)]{IMMM}. It is the \textit{Jordan degree\h{type}} of $A$ with respect to $\ell$.
\end{definition}

 We illustrate the definition with an example.

\begin{example}\label{exampleJordanTypeComputation}
Let $A=K[x,y]/(x^3,xy^2,y^3)$. A basis for $A$ is, for instance, 
\[
\{1,x,y,x^2,xy,y^2,x^2y\}.
\]
So the $h$\h{vector} is $(1,2,3,1)$ and $\dim_KA=7$. Consider the multiplication map 
    $$
    \times \ell : A \rightarrow A,
    $$
    where $\ell=x+y$. We have the following pre\h{Jordan} basis
\begin{equation*}
\begin{gathered}
\xymatrix@R=.1ex%@C=.5em
{
1 \ar@{|->}^-{\ell}[r] &x+y \ar@{|->}^-{\ell}[r] &x^2+2xy+y^2 \ar@{|->}^-{\ell}[r] &3x^2y \\
&x \ar@{|->}^-{\ell}[r] &x^2+xy \\
&& y^2
}
\end{gathered}
\end{equation*}
because \(x^2+xy\), when multiplied by \(\ell\), yields \(2x^2y\), which is a scalar multiple of \(3x^2y\), an element of the previous string; and in the third string we have $y^2\xmapsto{\ell}0$.
     But any pre\h{Jordan} basis yields a Jordan basis, as we already pointed out in Definition \ref{jordanBasisDefinition}. Indeed, although the last bead in the second string does not yield zero when multiplied by $\ell$, notice that $2\ell^3 - 3 \ell^2 x = 0$, i.e.\ $\ell^2(2\ell - 3x)=0$, so we could have considered, instead, the following string
    $$
    2\ell - 3x = 2y-x \xmapsto{\ell} 2y^2+xy-x^2 \xmapsto{\ell} 0.
    $$
    Thus, a Jordan basis is composed by the following strings
\[
(1,\ell,\ell^2,\ell^3),\, (2y-x,2y^2+xy-x^2), \text{ and }(y^2).
\]
Therefore $P_{\ell,A} = (4,2,1)$. Moreover, since the strings start at degrees \(0\), \(1\), and \(2\), the Jordan degree\h{type} is  
\[
\mathcal{S}_{\ell,A}=((4,0),(2,1),(1,2)),
\]
which can also be denoted $\mathcal{S}_{\ell,A}=(4_0,2_1,1_2)$.
\end{example}

Now let us define a partial order, called \text{dominance}, for partitions of positive numbers. 

\begin{definition}
\label{dominanceOrder}
Let $P=(p_1,\dots,p_s)$ and $Q=(q_1,\dots,q_r)$, with $p_1\ge \cdots \ge p_s$ and $q_1\ge \cdots \ge q_r$, be two partitions of a positive integer $n$. 
We say that $P$ dominates $Q$ (written $P\ge Q$) if, for each $k\in\{1,\dots,\min\{s,r\}\}$,
$$
\sum_{i=1}^k p_i \ge \sum_{i=1}^k q_i.
$$
\end{definition}

Given a graded Artinian algebra $A$, we will say that $\ell \in A_1$ is \textit{general} when it belongs to the open dense subset $U\subseteq A_1$ of linear forms whose Jordan type is the maximum with respect to the dominance order -- this definition is the reason we assumed that \(K\) is infinite. 
The \textit{generic linear Jordan type} of $A$, denoted $P_{A}$, is the Jordan type $P_{\ell,A}$ where $\ell $ is a general linear form (cf.\ \cite[Lemma 2.54, Definition 2.55]{IMMM}, where the notation for generic linear Jordan type is $P_{1,A}$, to distinguish it from the so called generic Jordan type of a general element of the maximal ideal). For instance, the form $\ell=x+y$ in Example \ref{exampleJordanTypeComputation} is general. 
Indeed, we know from Proposition 3.64 in \cite{HMMNWW} that the conjugate partition of the Hilbert function dominates the Jordan type of any linear form. Since the Hilbert function is \({(1,2,3,1)}\), its conjugate is \({(4,2,1)}\), which is the Jordan type of \({x+y}\).
%Certainly, due to dimensional reasons, $(4,3)$ is the only partition dominating $(4,2,1)$, and if there were a linear form whose Jordan type in $A$ were $(4,3)$, there would be a string commencing in degree one and terminating in the socle degree, which is absurd.

\begin{notation}
Whenever a length is repeated in a Jordan type, or in a Jordan degree\h{type}, we will use exponentiation in order to abbreviate notation. For instance, if 
    $$
    \mathcal{S}_{\ell,A}=((4,0),(4,1),(4,1),(3,1),(3,1),(1,3),(1,3)),
    $$
    we will write $P_{\ell,A}=(4^3,3^2,1^2)$ and $\mathcal{S}_{\ell,A}=(4_0,4_1^2,3_1^2,1_3^2)$.
\end{notation}

It is important to notice that computing all possible linear Jordan types for any Artinian algebra is not feasible. However, if we restrict to Perazzo algebras, which are Gorenstein, there is one paper (see \cite[Section 4.2]{A}) 
where the authors compute all the linear Jordan types for a certain kind of Perazzo algebras, among those associated to Perazzo threefolds in $\PP^4$ having the smallest possible $h$\h{vector}.

\subsection{Perazzo algebras}

\begin{definition}
\label{perazzohypersur} 
Fix integers $n, m\ge 2$. 
A \emph{Perazzo} hypersurface  ${X\subset \PP^{n+m}}$ of degree $d\ge 3$ is a hypersurface defined by a form
\[
F=X_0P_0+X_1P_1+\cdots +X_nP_n+G\in K[X_0,\dots ,X_n,Y_1\dots ,Y_m]_d,
\]
where \(P_0,\ldots,P_n\in K[Y_1,\dots ,Y_m]_{d-1}\) are algebraically dependent but linearly independent and \(G\in K[Y_1,\dots ,Y_m]_{d}\). By abuse of terminology, we call $F$ a \emph{Perazzo} form. A \emph{Perazzo} algebra $A_F$ is a standard graded Artinian Gorenstein algebra whose Macaulay dual generator is a Perazzo form $F$ of degree $d\ge 3$. 
\end{definition}
The linear independence of the forms $P_i$ guarantees that $V(F)$ is not a cone, while the algebraic dependence and the fact that $n+m\ge 4$ guarantee that the Hessian determinant of $F$ vanishes identically \cite{GN} (see \cite{L} and \cite{WB} for a modern version of Gordan and Noether's results). 
Perazzo hypersurfaces are the first known hypersurfaces with vanishing Hessian which are not cones.
Artinian Gorenstein algebras associated to Perazzo forms have been deeply studied from the viewpoint of the Lefschetz properties (check \cite[Definition 3.1 and Definition 3.8]{HMMNWW} for the definition). 
The importance of studying Perazzo algebras stems from the fact that they cannot have the strong Lefschetz property (SLP), because $\det(\Hess(F))\equiv 0$ (see \cite[Theorem 3.1]{MW}), but they can still satisfy the WLP. In \cite{A}, \cite{FMMR}, \cite{MMR23} and \cite{MP} the authors characterize certain classes of Perazzo algebras with the WLP.  

\begin{definition}
A Perazzo algebra $A_F$ is said to be \emph{full} if $\{P_0,\dots,P_n\}$ is a basis of $K[Y_1,\dots,Y_m]_{d-1}$. So, for full Perazzo algebras we necessarily have ${n+1=\binom{d+m-2}{m-1}}$. In this case, and in order to ease the comprehension of the main result in this paper, we will assume, after a change of variables, that $F=\sum_{i_1+\cdots+i_m=d-1}X_{i_1,\dots,i_m}Y_1^{i_1}\cdots Y_m^{i_m}$. 
\end{definition}

It is well known that, for any $d\ge 3$, the Hilbert function of a full Perazzo algebra $A_F$ is \cite[Proposition 3.3]{CGIZ}
\begin{equation}\label{HFfullPerazzoGeneral}
HF_{A_F}(i)=
\begin{cases}
    \binom{i+m-1}{m-1} + \binom{d-i+m-1}{m-1} \quad \text{for } 1\le i \le \lfloor d/2 \rfloor \\\\
    \text{symmetry}
\end{cases}.
\end{equation}
Thus,
\begin{align}\label{dimensionAF} 
    \dim_{K}A_F &= 2 + \sum_{i=1}^{d - 1} \Bigl( \tbinom{i+m-1}{m-1} + \tbinom{d-i+m-1}{m-1} \Bigr) \\
    &= 2+2\sum_{j=m}^{d+m-2}\tbinom{j}{m-1} \nonumber\\
    &= 2+2\Bigl( \tbinom{d+m-1}{m}-1 \Bigr) \nonumber \\
    &= 2\tbinom{d+m-1}{m}. \nonumber 
\end{align}

In particular, if $d=3$ we have
\begin{equation}\label{HFfullPerazzoParticularCases}
HF_{A_F}=(1,n+m+1,n+m+1,1) \quad \text{ and } \quad \dim_{K}A_F= 2(n+m+2).    
\end{equation}

\section{Jordan type of a full Perazzo algebra}\label{genericJordanTypeFPA}

In this section we compute all possible linear Jordan types in any full Perazzo algebra.
Let 
$$
R:=K[x_{i_1,i_2,\dots,i_m},y_1,\dots,y_m \ : \ i_1+\cdots+i_m=d-1]
$$ 
be a polynomial ring, and let 
$$
\D:=K[X_{i_1,i_2,\dots,i_m},Y_1,\dots,Y_m \ : \ i_1+\cdots+i_m=d-1]
$$
be a divided power algebra, regarded as an $R$\h{module} with the contraction action. Let $A_F=R/\Ann_RF$ be a full Perazzo algebra whose Macaulay dual generator ${F\in \mathcal{D}_d}$ is in the canonical form, i.e.\ $F=\sum_{i_1+\cdots+i_m=d-1}X_{i_1,\dots,i_m}Y_1^{i_1}\cdots Y_m^{i_m}$. The goal of this section is to compute the Jordan type of any linear form 
$$
\ell=\sum a_{i_1,\dots,i_m}x_{i_1,\dots,i_m} + \sum b_jy_j \in [A_F]_1.
$$
Let us start with a toy example.

\begin{example}
\label{toyExample}
Let $F=X_{2,0}Y_1^2+X_{1,1}Y_1Y_2+X_{0,2}Y_2^2\in \D_3$. Notice that $(x_{2,0},x_{1,1},x_{0,2})^2\subset \Ann_RF$ and for any $g\in K[y_1,y_2]_t$ we have that ${g\in \Ann_RF}$ if and only if ${t\ge 3}$. Moreover, by (\ref{HFfullPerazzoParticularCases}), we know that the $h$\h{vector} of $A_F:=R/\Ann_R F$ is $(1,5,5,1)$ and hence $\dim_K(A_F)=12$. 
Since we are assuming throughout the paper that the characteristic of \(K\) is either zero or higher than the socle degree, in this case it can be anything except \(2\) or \(3\).

We will prove that the possible Jordan types of a non\h{zero} form $\ell\in [A_F]_1$ in $A_F$ are in the following chain with respect to the dominance order.
$$
(2^3,1^6) < (2^4,1^4) < (3^2,2^2,1^2) < (4,2^3,1^2).
$$
In fact, writing \({\ell=a_{2,0}x_{2,0}+a_{1,1}x_{1,1}+a_{0,2}x_{0,2}+b_1y_1+b_2y_2}\), we have the following possibilities:

\begin{enumerate}[label=(\alph*)]
\item If $a_{2,0},b_1 \not= 0$ or $a_{0,2},b_2 \not=0$, we may assume, without loss of generality, that $a_{2,0},b_1\ne 0$. Then $P_{\ell,A_F}=(4,2^3,1^2)$ and $\mathcal{S}_{\ell,A_F}=(4_0,2_1^3,1_1,1_2)$. Indeed, $\ell^3= 3a_{2,0}b_1^2x_{2,0}y_1^2\not=0$ and hence there is exactly one string of length four. We have the following pre\h{Jordan} basis.
\begin{equation*}
\begin{gathered}
\xymatrix@R=.1ex%@C=.5em
{
    1 \ar@{|->}[r] & \ell \ar@{|->}[r] & \ell^2 \ar@{|->}[r] & \ell^3\\
    &x_{2,0} \ar@{|->}[r] & x_{2,0}\ell && \\
    &x_{1,1} \ar@{|->}[r] & x_{1,1}\ell && \\
    &y_2 \ar@{|->}[r] & y_2\ell && \\
    &x_{0,2} &&& \\
    && y_2^2. &&
}
\end{gathered}
\end{equation*}
\item If \(\ell={b_1y_1+b_2y_2}\), we may assume, without loss of generality, that \({b_1 \ne 0}\).  
Aslo, if \(\ell\) is either \({a_{1,1}x_{1,1}+a_{0,2}x_{0,2}+b_1y_1}\), with \({b_1 \ne 0}\), or \({a_{2,0}x_{2,0}+a_{1,1}x_{1,1}+b_2y_2}\), with \({b_2 \ne 0}\)), we may assume that \({\ell=a_{1,1}x_{1,1}+a_{0,2}x_{0,2}+b_1y_1}\), again with \({b_1 \ne 0}\). In any of these cases, we have \({P_{\ell,A_F}=(3^2,2^2,1^2)}\) and \({\mathcal{S}_{\ell,A_F}=(3_0,3_1,2_1^2,1_1,1_2)}\). Indeed, we have the following pre\h{Jordan} basis.
\begin{equation*}
\begin{gathered}
\xymatrix@R=.1ex%@C=.5em
{
1 \ar@{|->}[r] &\ell \ar@{|->}[r] &\ell^2 \\
&x_{2,0} \ar@{|->}[r] &x_{2,0}\ell \ar@{|->}[r] &x_{2,0}\ell^2 \\
&x_{1,1} \ar@{|->}[r] &x_{1,1}\ell & \\
&y_2 \ar@{|->}[r] &y_2\ell & \\
&x_{0,2} && \\
&& y_2^2. &
}
\end{gathered}
\end{equation*}
To check this, observe that if $\ell=a_{1,1}x_{1,1}+a_{0,2}x_{0,2}+b_1y_1$ it is clear that all these elements are linearly independent; if $\ell=b_1y_1+b_2y_2$ they are also linearly independent because
\begin{align*}
b_1^2X_{2,0}+2b_1b_2X_{1,1}+b_2^2X_{0,2} &= \ell^2\circ F\\
b_1X_{1,1}+b_2X_{0,2} &= y_2\ell\circ F \\
X_{0,2} &= y_2^2\circ F.
\end{align*}
\item If $\ell=a_{2,0}x_{2,0}+a_{1,1}x_{1,1}+a_{0,2}x_{0,2}$, then $P_{\ell,A_F}=(2^a,1^b)$, where $2a + b=\dim_K(A_F)$ and $a=\dim_K(A_{\ell\circ F})$. Indeed, $\ell\circ F=a_{2,0}Y_1^2+a_{1,1}Y_1Y_2+a_{0,2}Y_2^2$ and the $h$\h{vector} of $A_{\ell\circ F}$ is $(1,r,1)$, where $r$ is given by the dimension of the vector subspace of partials of $\ell\circ F$ of degree one, due to Macaulay\h{Matlis} duality. So
$$
r=rk\begin{pmatrix} a_{2,0} & a_{1,1} \\ a_{1,1} & a_{0,2}\end{pmatrix} =\begin{cases} 1 \ \text{if } a_{2,0}a_{0,2} - a_{1,1}^2=0 \\ 2 \ \text{otherwise}. \end{cases}
$$
Now, since there is no string of length one starting with a vector in $A_{\ell\circ F}$, the dimension of $A_{\ell\circ F}$ must be the number of length two strings, because $\ell^2\subset \Ann_R F$. Therefore, either $P_{\ell,A_F}=(2^4,1^4)$ and $\mathcal{S}_{\ell,A_F}=(2_0,2_1^2,2_2,1_1^3,1_2)$ or $P_{\ell,A_F}=(2^3,1^6)$ and $\mathcal{S}_{\ell,A_F}=(2_0,2_1,2_2,1_1^6)$.    
\end{enumerate}
\end{example}

We are now ready to state the main result of this paper.

\begin{theorem}\label{mainTheorem}
    Let $A_F=R/\Ann_RF$ be a full Perazzo algebra, with 
    $$
    F=\sum_{i_1+\cdots+i_m=d-1}X_{i_1,\dots,i_m}Y_1^{i_1}\cdots Y_m^{i_m},
    $$ 
    and consider any linear form
    $$
    \ell = \sum_{i_1+\cdots+i_m=d-1} a_{i_1,\dots,i_m}x_{i_1,\dots,i_m} + \sum_{j=1}^m b_jy_j \in [A_F]_1.
    $$
    Then the Jordan type of \(\ell\) in \(A_F\) is one of the following
\begin{enumerate}[label=(\roman*)]
\item \label{jordanTypeMax} $P_{\ell,A_F} = \Bigl( d+1, (d-1)^{2m-1},(d-2)^{2\binom{m}{m-2}},\dots,1^{2\binom{m+d-3}{m-2}} \Bigr)$, or
\item \label{jordanTypeMiddle} $P_{\ell,A_F} = \Bigl( d^2, (d-1)^{2\binom{m-1}{m-2}},(d-2)^{2\binom{m}{m-2}},\dots,1^{2\binom{m+d-3}{m-2}} \Bigr)$, or
\item \label{jordanTypeMin} $P_{\ell,A_F} = \bigl( 2^a,1^b \bigr)$ where $2a+b=\dim_KA_F$ and $a=\dim_KA_{\ell\circ F}$. 
\end{enumerate}
Furthermore, in cases \ref{jordanTypeMax}  and \ref{jordanTypeMiddle} the corresponding Jordan degree\h{type}s are
\begin{align*}
\mathcal{S}_{\ell,A_F} &= \Bigl( (d+1)_0, (d-1)_1^{2m-1},(d-2)_1^{\binom{m}{m-2}},(d-2)_2^{\binom{m}{m-2}},\dots,1_1^{\binom{m+d-3}{m-2}},1_{d-1}^{\binom{m+d-3}{m-2}}\Bigr), \text{ and} \\ \mathcal{S}_{\ell,A_F} &= \Bigl( d_0,d_1, (d-1)_1^{2\binom{m-1}{m-2}},(d-2)_1^{\binom{m}{m-2}},(d-2)_2^{\binom{m}{m-2}},\dots,1_1^{\binom{m+d-3}{m-2}},1_{d-1}^{\binom{m+d-3}{m-2}}\Bigr)
\end{align*}
respectively.
\end{theorem}

\begin{proof}
Notice that $(x_{i_1,\dots,i_m}\ : \ i_1+\cdots+i_m=d-1)^2\subset \Ann_R F$ and for $g\in k[y_1,\dots,y_m]_t$ we have that $g\in\Ann_RF$ if and only if $t\ge d$. 
We have the following possibilities:

\vskip 2mm

(i) \label{genJordanType} If, for any $k\in\{1,\dots,m\}$, $a_{\substack{0,\dots,0,d-1,0,\dots,0 \\ (k)}}\not=0$ and $b_k\not=0$, we may assume, without loss of generality, that $a_{d-1,0,\dots,0}\not=0$ and $b_1\not=0$. Then 
$$
P_{\ell,A_F} = \Bigl( d+1, (d-1)^{2m-1},(d-2)^{2\binom{m}{m-2}},\dots,1^{2\binom{m+d-3}{m-2}} \Bigr).
$$
Indeed, a pre\h{Jordan} basis is composed by the following strings.
\begin{align}
\label{jordanTypeMaxString1}
    &(1,\ell,\dots,\ell^{d})\\
\label{jordanTypeMaxString2}
    &(x_{d-1,\dots,0},\dots,x_{d-1,\dots,0}\ell^{d-2})\\
\label{jordanTypeMaxString3}
    &\{(x_{p,\dots,i_m},\dots,x_{p,\dots,i_m}\ell^p) : p\in\{0,\dots,d-2\}\}\\
\label{jordanTypeMaxString4}
    &{\textstyle \bigl\{(y_2^{j_2}\cdots y_m^{j_m},\dots,y_2^{j_2}\cdots y_m^{j_m}\ell^p) : p\in\{0,\dots,d-2\}; \sum_{i\not= 1}j_i=d-p-1\bigr\}}.
\end{align}
Indeed, all these vectors are linearly independent because there is no degree such that the forms involved are obtained from the same monomials. Now, counting  the beads in all these strings, we obtain        
        \begin{equation}\label{numberBeadsGeneric}
        2\sum_{i=1}^{d}i\tbinom{m+d-2-i}{m-2} = 2\tbinom{m+d-1}{m} = \dim_K(A_F) \quad \eqref{dimensionAF}    
        \end{equation}        
        where the possible ways of choosing $m$ integers from the set $\{1,2,\dots,m+d-1\}$ is computed as follows: given that $m\ge 2$, if the second smallest number we choose is $i+1$, then we can choose the smallest number in $i$ ways. The remaining $m-2$ numbers can be chosen in $\tbinom{m+d-2-i}{m-2}$ ways, so that gives $i\tbinom{m+d-2-i}{m-2}$ combinations. Now, sum over all possible $i$. 
        
        Finally, let us check the Jordan degree\h{type}. We may observe that the string \eqref{jordanTypeMaxString1} has length $d+1$ and starts at degree zero, which justifies the first part $(d+1)_0$. Next, we see that if \(p=d-2\) in \eqref{jordanTypeMaxString3}, we have \(m-1\) strings of length $d-1$, starting at degree one, i.e.\ the strings starting by the beads \(x_{d-2,1,0,\ldots,0}\) up to \(x_{d-2,0,\ldots,0,1}\). Moreover, the string \eqref{jordanTypeMaxString2} has also length \(d-1\). 
        If \(p=d-2\) in \eqref{jordanTypeMaxString4}, we also have \(m-1\) strings of degree $d-1$, starting at degree one, i.e.\ the strings starting by the beads \(y_2\) up to \(y_m\).
        Thus, we get a total of \(2m-1\) strings of length \(d-1\), starting at degree one, which explains the parts \((d-1)_1^{2m-1}\).
        Note that, for strings of length \(d-2\), we find \(\tbinom{m}{m-2}\) strings in \eqref{jordanTypeMaxString3}, starting at degree one, and another \(\tbinom{m}{m-2}\) strings in \eqref{jordanTypeMaxString4}, but starting at degree two (we take \(p=d-3\) in both cases).  We may check the remaining parts of the Jordan degree\h{type} in a similar way.

\vskip 2mm
(ii) Suppose that either 
\begin{itemize}
\item[-] there is an integer \({k\in\{1,\dots,m\}}\) such that \({a_{\substack{0,\dots,0,d-1,0,\dots,0 \\ (k)}}=0}\), \({b_k \not=0}\), and \({b_i=0}\) for all \({i \not=k}\), or
\item[-] we have \({\ell=b_1y_1+\cdots+b_my_m}\).
\end{itemize}
Then 
\[
P_{\ell,A_F} = \Bigl( d^2, (d-1)^{2\binom{m-1}{m-2}},(d-2)^{2\binom{m}{m-2}},\dots,1^{2\binom{m+d-3}{m-2}} \Bigr).
\]
Indeed, we may assume, without loss of generality, that 
\begin{itemize}
\item[-] \({a_{d-1,0,\dots,0}=0}\), \({b_1 \not=0}\), and \({b_i=0}\) for all \({i >1}\), or
\item[-] we have \({\ell=b_1y_1+\cdots+b_my_m}\), with \({b_1 \not=0}\).
\end{itemize}
We can then check that a pre\h{Jordan} basis is composed by the following strings.
        \begin{align}
        \label{jordanTypeMiddlestring1}
            &(1,\ell,\dots,\ell^{d-1})\\
        \label{jordanTypeMiddlestring2}
            &(x_{d-1,\dots,0},\dots,x_{d-1,\dots,0}\ell^{d-1})\\
        \label{jordanTypeMiddlestring3}
            &\{(x_{p,\dots,i_m},\dots,x_{p,\dots,i_m}\ell^p) : p\in\{0,\dots,d-2\}\}\\
        \label{jordanTypeMiddlestring4}
            &{\textstyle \bigl\{(y_2^{j_2}\cdots y_m^{j_m},\dots,y_2^{j_2}\cdots y_m^{j_m}\ell^p) : p\in\{0,\dots,d-2\}; \sum_{i\not= 1}j_i=d-p-1\bigr\}}.
        \end{align}
        Certainly, if $a_{d-1,0,\dots,0}=0$, $b_1\not=0$ and $b_i=0$ for any $i\not=1$, then  all these vectors are linearly independent. If $\ell=b_1y_1+\cdots+b_ky_k$, then $\ell^{d-1}\circ F$ and $y_2^{j_2}\cdots y_m^{j_m}\ell^p\circ F$, where $p\in\{0,\dots,d-2\}$ and $\sum_{i\not= 1}j_i=d-p-1$, are linearly independent because, taking the basis $\{ X_{i_1,\dots,i_m} : i_1+\cdots+i_m=d-1 \}$ in lexicographic order, no two of them have their first non\h{zero} coordinate in the same position. Adding the beads in all these strings, we obtain the same as in \eqref{numberBeadsGeneric}. Finally, we may check the Jordan degree\h{type} as we did for the case (i). 
\vskip 2mm
(iii) \label{formsx} If $\ell=\sum_{i_1+\cdots+i_m=d-1}a_{i_1,\dots,i_m}x_{i_1,\dots,i_m}$, we have $\ell ^2\subset \Ann_R F$ and therefore there are only strings of lengths one and two, which means that the Jordan type is
$P_{\ell,A_F} = \bigl( 2^a,1^b \bigr)$ where $2a+b=\dim_KA_F$ and $a=\dim_KA_{\ell\circ F}$. In order to obtain $P_{\ell,A_F}$ it suffices to compute the number of length two strings, that is, $\dim_K(A_{\ell\circ F})$.        
  
\end{proof}

\begin{corollary}\label{partsOfGenericJordanTypeFullPerazzo}
    The number of parts in a generic linear Jordan type of a full Perazzo algebra is $2(n+1)$.
    \end{corollary}

    \begin{proof}
       Indeed, by Theorem \ref{mainTheorem} (i), we have: 

    \begin{align*}
        1 + 2m-1 + 2 \sum_{i=0}^{d-3}\tbinom{m+i}{m-2} &= 2m + 2\sum_{i=m}^{d+m-3}\tbinom{i}{m-2} \\
        &= 2m + 2\biggl( \sum_{i=0}^{d+m-3} \tbinom{i}{m-2} - \sum_{i=0}^{m-1} \tbinom{i}{m-2} \biggr) \\
        &= 2m + 2\Bigl( \tbinom{d+m-2}{m-1} - \tbinom{m}{m-1}\Bigr) \\
        &= 2m + 2 (n+1-m) \\
        &= 2(n+1).
    \end{align*}
\end{proof}
    
\begin{corollary}\label{corollaryRangeJordanType}
    The Jordan types presented in Theorem \ref{mainTheorem} are in the following chain with respect to the dominance order:
    \begin{align*}
        \bigl( 2^{a_{min}},1^{b_{max}} \bigr) < \cdots &< \bigl( 2^{a_{max}},1^{b_{min}} \bigr) \\
        &< \bigl( d^2, (d-1)^{2\binom{m-1}{m-2}},(d-2)^{2\binom{m}{m-2}},\dots,1^{2\binom{m+d-3}{m-2}} \bigr) \\         
        &< \bigl( d+1, (d-1)^{2m-1},(d-2)^{2\binom{m}{m-2}},\dots,1^{2\binom{m+d-3}{m-2}} \bigr),
    \end{align*}
    where $a_{min}=d$,
    \begin{equation}\label{aMax}
    a_{max}=
    \begin{cases}
        \binom{\frac{d-1}{2}+m-1}{m-1} + 2\sum_{i=0}^{\frac{d-3}{2}}\binom{i+m-1}{m-1} \quad \text{if } d \text{ is odd} \\\\
        2\sum_{i=0}^{\frac{d-2}{2}}\binom{i+m-1}{m-1} \quad \text{if } d \text{ is even},
    \end{cases}
    \end{equation}
    $b_{max}=\dim_K(A_F)-a_{min}$, $b_{min}=\dim_K(A_F)-a_{max}$.
\end{corollary}

\begin{proof}
    The chain is clear, due to the Jordan types just obtained. Now, notice that $A_{\ell\circ F}$ is Gorenstein \cite[Proposition 3.1]{MP}, Artinian and with socle degree $d-1$. So, in the case \ref{jordanTypeMin} there exists $G\in K[Y_1,\dots,Y_m]_{d-1}$ such that $A_{\ell\circ F}\cong \frac{K[y_1,\dots,y_m]}{\Ann_{K[y_1,\dots,y_m]}G}$. Moreover, for any $h$\h{vector} associated to such an algebra there exists 
    $$
    \ell=\sum_{i_1+\cdots+i_m=d-1}a_{i_1,\dots,i_m}x_{i_1,\dots,i_m}
    $$
    such that $\ell\circ F=G$ (take $F=LG$, where $\ell\circ L=1$). So, there is $G_{max}$ whose associated Gorenstein algebra $A_{G_{max}}$ is a compressed Gorenstein algebra of codimension $m$ and socle degree $d-1$ (take for instance the homogeneous symmetric polynomial of degree $d-1$, i.e.\ $G_{max}:=h_{d-1}(Y_1,\dots,Y_m)$) and hence the number of length two strings is the cited in \eqref{aMax}. Also there exists a form $G_{min}$ whose associated Gorenstein algebra $A_{G_{min}}$ (take for instance $G_{min}=Y_m^{d-1}$) has the minimum $h$\h{vector}, i.e.\ $(1,1,\dots,1)$, and hence $a_{min}=d$.    
\end{proof}

In particular, for $m=2$ we have the following result.

\begin{proposition}
Let 
$$
F=X_{d-1,0}Y_1^{d-1}+X_{d-2,1}Y_1^{d-1}Y_2+\cdots+X_{1,d-2}Y_1Y_2^{d-2}+X_{0,d-1}Y_2^{d-1} \in \mathcal{D}_d
$$
    and consider the linear form
    $$
    \ell=\sum_{i=0}^{d-1}a_{d-1-i,i}x_{d-1-i,i}\in [A_F]_1.
    $$
    Then the $h$\h{vector} of $A_{\ell\circ F}$ is
    $$
    (1,2,\dots,s-1,s,s,\dots,s,s,s-1,\dots,2,1),
    $$
    where
    $$
    s=rk\begin{pmatrix}
        a_{d-1,0} & a_{d-2,1} & \cdots & a_{d-1-r,r} \\
        a_{d-2,1} & a_{d-3,2} & \cdots & a_{d-2-r,r+1} \\
        \vdots & \vdots & \ddots & \vdots \\
        a_{r,d-1-r} & a_{r-1,d-r} & \cdots & a_{0,d-1}
    \end{pmatrix}
    $$
    with
    $$
    r:=\begin{cases}
        \frac{d-1}{2} \ \text{if } d \text{ is odd} \\
        \frac{d-2}{2} \ \text{if } d \text{ is even}.
    \end{cases}
    $$
\end{proposition}
\begin{proof}
Notice that $\Ann_R(\ell\circ F)\in K[y_1,y_2]$ and hence $A_{\ell\circ F}$ is a complete intersection because it is Artinian Gorenstein and has codimension two. 
As a consequence, the $h$\h{vector} is the cited one, with $s=\dim_K\bigl[ A_{\ell\circ F} \bigr]_r$. 
Now, using Macaulay\h{Matlis} duality, we can compute $s$ as the dimension of the vector subspace of partials of $\ell\circ F$ of degree $r$, i.e.\     
    $$s = \dim
    \begin{amatrix}
    a_{d-1,0}Y_1^r + a_{d-2,1}Y_1^{r-1}Y_2 + \cdots + a_{d-1-r,r}Y_2^r, \\
    a_{d-2,1}Y_1^r + a_{d-3,2}Y_1^{r-1}Y_2 + \cdots + a_{d-2-r,r+1}Y_2^r, \\
    \vdots \\
    a_{r,d-1-r}Y_1^r + a_{r-1,d-r}Y_1^{r-1}Y_2 + \cdots + a_{0,d-1}Y_2^r
    \end{amatrix}.
    $$    
\end{proof}
\begin{remark}
    Let $M$ be the projective space $\PP^{(d-r)(r+1)-1}$ associated to the vector space of $(d-r)\times(r+1)$ matrices. Then, $P_{\ell,A_F}$ is completely characterized by the values of $a_{d-1-i,i}$ in $\ell$, that is, the corresponding point in the determinantal variety $M_s \subset M$, cut out by the $(s+1)\times(s+1)$ minor determinants of the matrices in $M$.
\end{remark}

\begin{remark}
For \(m=3\), the possible Hilbert functions for graded Artinian Gorenstein algebras are also known (see Proposition~3.9 in \cite{D}, and Theorem~B.17 in \cite{IK}). 
\end{remark}

\printbibliography

\end{document}